\documentclass[a4paper]{amsart}
\usepackage{amsfonts}
\usepackage{amssymb}
\usepackage{enumerate}
\usepackage{color}
\usepackage{nicefrac}
\usepackage{url}
\usepackage{color}
\usepackage[dvips]{graphicx}

\newcommand{\R}{\mathbb{R}}
\newcommand{\Rn}{\mathbb{R}^N}

\newcommand{\ep}{\varepsilon}
\newcommand{\B}{\mathcal{A}}
\newcommand{\Ae}{\mathcal{A}_{\ep}}
\newcommand{\Es}{\mathbb{E}_{S_{I},S_{II}}^{(x_0,t_0)}}
\newcommand{\Esy}{\mathbb{E}_{S_{I}^x,S_{II}}^{(x,t)}}
\newcommand{\dis}{\displaystyle}
\newcommand{\SN}{\mathbb{S}^{N}}

\newtheorem{te}{Theorem}[section]
\newtheorem{lem}[te]{Lemma}
\newtheorem{co}[te]{Corollary}

\theoremstyle{remark}
\newtheorem{re}[te]{Remark}

\theoremstyle{definition}
\newtheorem{de}[te]{Definition}

\theoremstyle{example}
\newtheorem{exa}[te]{Example}

\title[Tug-of-War games and parabolic problems]{Tug-of-War games and parabolic problems with spatial and time dependence}
\author[L. M. Del Pezzo and J. D. Rossi]
{Leandro M. Del Pezzo and Julio D. Rossi}

\address{Leandro M. Del Pezzo \hfill\break\indent
CONICET and Departamento  de Matem{\'a}tica, FCEyN, Universidad de Buenos Aires,
\hfill\break\indent Pabellon I, Ciudad Universitaria (1428),
Buenos Aires, Argentina.}

\email{{\tt ldpezzo@dm.uba.ar}}

\address{Julio D. Rossi \hfill\break\indent
Departamento  An{\'a}lisis Matem{\'a}tica, Universidad de Alicante,
\hfill\break\indent Ap. correo 99, 03080, Alicante, Spain.
}

\email{{\tt julio.rossi@ua.es}}
\thanks{Leandro M. Del Pezzo was supported by CONICET (Argentina) PIP 5478/1438 tand Julio D. Rossi by MTM2008-05824, (Spain).    }

\begin{document}
\maketitle

\begin{abstract}
In this paper we use probabilistic arguments (Tug-of-War games) to
obtain existence of viscosity solutions to a parabolic problem of
the form
$$
\begin{cases} K_{(x,t)}(D u )u_t (x,t)= \frac12 \langle
D^2 u  J_{(x,t)}(D u ),J_{(x,t)}(D u) (x,t)\rangle
 &\mbox{in } \Omega_T,\\
  u(x,t)=F(x)&\mbox{on }\Gamma,
\end{cases}
$$
where $\Omega_T=\Omega\times(0,T]$  and
$\Gamma$ is its parabolic boundary.
This problem can be viewed as a version with spatial and time dependence of
the evolution problem given by the infinity Laplacian,
$ u_t (x,t)=  \langle
D^2 u (x,t) \frac{D u}{|Du|} (x,t),\, \frac{D u}{|Du|}
(x,t)\rangle$.
\end{abstract}

\section{Introduction}

Our goal in this article is to look for parabolic PDEs that may
arise as continuous values of Tug-of-War games when one takes into
account the number of plays that the players play and considering
sets of possible movements that may depend on space and time. In
this way we obtain what we can call a natural way of defining a
{\it parabolic problem involving the infinity Laplacian with
spatial and time dependence}.

\medskip

Solutions to the infinity Laplacian
$	
 \langle D^2 u (x,t) \frac{D u}{|Du|} (x,t),\,
\frac{D u}{|Du|} (x,t)\rangle=0
$
appear naturally when one
considers absolutely minimizing Lipschitz extensions (AMLE) of a
Lipschitz function $F$ defined on the boundary; see the survey
\cite{ACJ} and \cite{Jensen} (see also \cite{AM,ChP,JL,JLMeigen, JLM}).
This equation (and also the $p-$Laplacian) was related
to continuous values of Tug-of-War games, see \cite{PSSW}. See
also \cite{AM,BEJ,ChGAR,KS,mpr-dpp,mpr-definition,mpr-parabolic, Peres,PS} and, for numerical approximations, \cite{Oberman}.

The evolution problem given by the infinity Laplacian, is given by
\begin{equation}\label{yo}
v_t (x,t)=  \Big\langle D^2 v (x,t) \frac{D v}{|Dv|} (x,t),\,
\frac{D v}{|Dv|} (x,t)\Big\rangle.
\end{equation}
For existence, asymptotic behaviour and further properties of the solutions we refer to \cite{juu2,juu}.

Recently, see \cite{mpr-parabolic}, probabilistic methods (based
on Tug-of-War games) where used to obtain mean value
characterizations of solutions to parabolic PDEs, including the equation \eqref{yo}.
The Tug-of-War game related to this equation, see
\cite{mpr-parabolic}, can be briefly described as follows:
a Tug-of-War game is a two-person, zero-sum game, that
is, two players are in contest and the total earnings of one are
the losses of the other.
Let $T$ be a positive constant and $\Omega$ be a bounded smooth
open subset of $\Rn$. We consider the parabolic cylinder
$\Omega_T=\Omega\times(0,T]$ with the parabolic boundary
$\Gamma=\partial\Omega\times[0,T]\cup\Omega\times\{0\}$ and, for a
fixed $\eta>0$ we define a strip around the parabolic boundary
$
\Gamma_\eta =\Omega_\eta\times(-\eta^2,0]\cup\Theta_\eta\times(0,T],
$
where
$
\Omega_\eta=\left\{x\in\Rn  \colon \mbox{dist}(x,\Omega)\le \eta \right\}
$
and
$
\Theta_{\eta}=\{x\in\Rn\setminus\Omega\colon\mbox{dist}(x,\partial\Omega)\le
\eta\}$.
Let $F:\Gamma_\eta \rightarrow\R$ be a Lipschitz continuous
function (the final payoff function).
The rules of the game are the following: At the initial time,
$t_0$, a token is placed at a point $x_0\in\Omega$. Then, a (fair)
coin is tossed and the winner of the toss is allowed to move the
game position to any $x_1\in\overline{B_\epsilon(x_0)}$ and the
time is decreased by $c\epsilon^2$ ($c$ is just a normalizing
constant, see \cite{mpr-parabolic}). At each turn, the coin is
tossed again, and the winner chooses a new game state $x_k\in
\overline{B_\epsilon(x_{k-1})}$ while time decreases at each time $c\epsilon^2$.
Once the token has reached some $(x_\tau, \tau)\in\Gamma_\eta$,
the game ends and the first player earns $F(x_\tau, \tau)$ (while
the second player earns $-F(x_\tau, \tau)$). This game has a
expected value $u_\epsilon (x_0)$ (called the value of the game)
that verifies the Dynamic Programming Principle (DPP),
\begin{equation}\label{DPP.bolas}
 u_\epsilon (x,t)=\frac12 \sup_{y\in
\overline{B_\epsilon(x)}}u_\epsilon(y, t-c\epsilon^2) +\frac12 \inf_{y\in
\overline{B_\epsilon(x)}}u_\epsilon(y, t-c\epsilon^2)
\end{equation}
for every $(x,t)\in \Omega\times (0,T)$. In the above equation, it
is understood that $u_\epsilon (x,t) = F(x,t)$ for
$(x,t)\in\Gamma_\eta$. This formula can be intuitively explained
from the fact that the first player tries to maximize the expected
outcome (and has probability $1/2$ of selecting the next state of
the game) while the second tries to minimize the expected outcome
(and also has probability $1/2$ of choosing the next position).
As $\epsilon \to 0$ we have that
$
u_\epsilon \rightrightarrows v
$
uniformly and this limit $v$ (that is called the continuous value
of the game) turns out to be a viscosity solution to \eqref{yo}
with the Dirichlet boundary condition $v(x,t) = F(x,t)$, for $(x,t) \in \Gamma$.
The fact that the limit is a solution to the equation can be
intuitively explained as follows: for a smooth function $\phi$
with non-zero gradient the maximum in $\overline{B_\epsilon(x)}$
is attained at a point on the boundary of the ball $\partial
{B_\epsilon(x)}$ that lies close to the direction of the gradient,
that is, the location of the maximum is close to $x+
\epsilon D\phi(x) / |D\phi(x)|$. Analogously the minimum is close
to $x -\epsilon D\phi(x) / |D\phi(x)|$ and hence the DPP, equation
\eqref{DPP.bolas}, for the smooth function $\phi$ reads as
$$
\begin{array}{l}
\displaystyle
\phi(x,t) - \phi (x, t-c\epsilon^2) \sim \frac12 \phi \Big(x+ \epsilon \frac{D\phi(x, t-c\epsilon^2) }{
|D\phi(x, t-c\epsilon^2) |}, t-c\epsilon^2\Big)\\[12pt]
\qquad \qquad \displaystyle  +\frac12
\phi \Big(x -
\epsilon \frac{ D\phi (x, t-c\epsilon^2) }{ |D\phi(x, t-c\epsilon^2) |},
t-c\epsilon^2\Big)- \phi (x, t-c\epsilon^2),
\end{array}
$$
that is a discretization of the equation. Note that the right hand
side is a discretization of the second derivative in the direction
of the gradient. This formal calculation can be fully justified
when one works in the viscosity sense, see \cite{mpr-parabolic}.



\medskip

As we have mentioned, our goal in this paper is to show that one
can obtain existence of viscosity solutions to more general
parabolic equations when one allows the possible movements of the
players. To be more precise,  our main concern in this paper is to answer the following
question:

\medskip

{\it What are the PDEs that can be obtained as continuous values
of Tug-of-War games when we replace the ball $\overline{B_\epsilon
(x)}$ with a more general family of sets $\mathcal{A}_\epsilon
(x,t)$ ?}

\medskip

To answer this question we have to assume certain conditions on
the family of sets $\mathcal{A}_\epsilon (x,t)$ and the way that
they behave as $\epsilon \to 0$ (see Section \ref{dotg} for
details). If we play the same game described before with the possible positions
given by the sets $\mathcal{A}_\epsilon (x,t)$  the DPP reads as
$$
u_{\epsilon}(x,t)=\frac{1}{2}
\sup_{(y,s)\in \mathcal{A}_\epsilon(x,t)}u_{\epsilon}(y,s)
+ \frac{1}{2} \inf_{(y,s)\in
\mathcal{A}_\epsilon (x,t)}u_{\epsilon}(y,s).
$$
Following our previous discussion for the case of balls we can
guess that the limit PDE as $\epsilon \to 0$ will depend on the
point at which a smooth function $\phi$ with non-zero gradient
attains its maximum (and its minimum) in
$\mathcal{A}_\epsilon(x,t)$. Our conditions on the sets
$\mathcal{A}_\epsilon(x,t)$ are such that there is a preferred
direction where the maxima and the minima of a smooth function
$\phi$ with non-zero gradient are closely located when $\epsilon
\to 0$. This preferred direction depends on the spatial location
and on the gradient of $\phi$ at that point. We call such
direction $J_{(x,t)} (D\phi (x,t))$.
Also, due to scaling properties of the sets there is a preferred time that depends on $x,t$ and $D \phi$, we call it  $K_{(x,t)}(D\phi(x,t))$. With this in mind our main result reads as follows:

\medskip

{\it Under adequate assumptions on the family of sets $\mathcal{A}_\epsilon(x,t)$ there is a uniform limit
(along a subsequence) as $\epsilon \to 0$ of the values of the game, $v$, that is a viscosity solution to
$$
K_{(x,t)}(Du(x,t))u_t (x,t)= \dfrac12 \langle
D^2 u (x,t) J_{(x,t)}(D u (x,t)),J_{(x,t)}(D u (x,t))\rangle
$$
in $ \Omega_T= \Omega \times (0,T]$, with boundary condition
 $ u(x,t)=F(x)$ on the parabolic boundary $\Gamma$.
}

Uniqueness for this general problem and regularity issues seem delicate and are left open.

\medskip

{\bf Organization of the paper.} In Section \ref{dotg}, we describe with some details
the Tug-of-War game, introduce the precise set conditions that we assume on the family
of the set $\Ae(x,t),$ state the DPP for our game and prove that the game has a value and the comparison principle for values of the game;
in Section \ref{uniconv} we prove that the $\ep-$value of the game converge uniformly to a continuous function; finally, in Section \ref{sect.eq.limite} we show that the limit is a viscosity solution to our parabolic equation.

\medskip

Throughout this paper, the points in $\Rn$ are denoted by
$x=(x^1,\dots,x^N),$ $|\cdot|$ denote the $2-$norm in $\Rn,$
$\langle \cdot,\cdot \rangle $ denote the usual inner product of
$\R^N,$ the ball of center $x_0\in\Rn$ and radius $\rho>0$  is
denoted by $B(x_0,\rho)$  and $\pi_1,\pi_2:\R^{N+1}\to\R$ denote
the projections with respect to the $x-$axis and $t-$axis
respectively. Finally, let $\SN$ denote the space of symmetric $N\times N$
matrices.

\section{Description of the game}\label{dotg}

Now, we describe the Tug-of-War game, following \cite{mpr-dpp,mpr-parabolic}.

Let $F:\Gamma_\eta\to\R$ be a bounded Borel function,
$F$ is called the final payoff function.

{\bf Tug-of-War game with spatial and time dependence.} A Tug-of-War game is a zero-sum game between two players (Player I and Player II). At the beginning a token is
placed at a point $(x_0,t_0)\in\Omega_T$ and we fix $\ep>0.$ Then, the players toss a fair coin and the winer decides a new game state
$(x_1,t_1)$ in a set $\Ae(x_0,t_0),$ that depends on the position $(x_0,t_0)$ and will be defined later.
Then, the coin is tossed again and the winer chooses a new game state $(x_2,t_2)\in \Ae(x_1,t_1).$ They continue playing the game
until the token hits the parabolic boundary strip $\Gamma_\ep.$ At the end of the game, Player II pays Player I the amount given by the
payoff function $F$, that is, Player I earns
$F(x_{\tau},t_{\tau})$ and the Player II earns $-F(x_{\tau},t_{\tau}),$  where $\tau$
is the number of rounds (a stopping time) that takes the game to end.  Later, we will show that $0<\tau<+\infty$ (see Remark \ref{tiempofinito}).
This procedure yields a sequence of game states $ (x_0,t_0),(x_1,t_1),\dots,(x_{\tau},t_\tau),$ where every $(x_k,t_k)$ except $(x_0,t_0)$
are random variables, depending on the coin tosses and the strategies adopted by the players.
A strategy $S_I$ for Player I is a collection of measurable mappings $S_I=\{S_I^k\}_{k=1}^\tau$ such that the next game position is
\[
S_I^{k+1} ( (x_0,t_0),(x_1,t_1),\dots,(x_k,t_k))=(x_{k+1},t_{k+1}) \in\Ae( x_k,t_k),
\]
if Player I  wins the coin toss given the partial history $((x_0,t_0),(x_1,t_1),\dots,(x_{\tau},t_\tau)).$
Similarly, Player II plays according to the strategy $S_{II}.$ The next game position  $(x_{k+1},t_{k+1})\in \Ae( x_k,t_k),$ given the
history $((x_0,t_0),(x_1,t_1),\dots,(x_k,t_k)),$ is selected according to a probability distribution
$p(\cdot|(x_0,t_0),(x_1,t_1),\dots,(x_k,t_k))$ which, in our case, is given by the fair coin toss.

The fixed starting point $(x_0,t_0),$ the domain $\Omega_T$ and the strategies $S_I$ and $S_{II}$ determine a unique probability measure
$\mathbb{P}_{S_I, S_{II}}^{x_0}$ on the space of plays $(\Omega_T\cup\Gamma_\ep)^{\infty}.$ We denote by $\Es$ the corresponding expectation.
If $S_I$ and $S_{II}$ denote the strategies adopted by the Player I and II respectively, given $(x_0,t_0)\in\Omega_T,$ the expected payoff is given by
$
\Es[F(x_\tau,t_\tau)].
$
The $\ep-$value for the Player I, when starting from $(x_0,t_0),$ is then defined as
\[
u_I^\ep(x_0,t_0)=\sup_{S_I}\inf_{S_{II}}\Es[F(x_\tau,t_\tau)],
\]
while the $\ep-$value of the game for the Player II is given by
\[
u_{II}^\ep(x_0,t_0)=\inf_{S_{II}}\sup_{S_{I}}\Es[F(x_\tau,t_\tau)].
\]

Now, we will describe the family of subsets
of $\Omega_T\cup\Gamma_\ep$ that encode the possible movements of the game.

We consider a family of sets $\{\B(x,t)\}_{(x,t)\in\Omega_T}$ with the following
properties: For every $(x,t)\in\Omega_T,$

\begin{enumerate}
      \item[{\bf A1.}] $\B(x,t)$ is a compact subset of
         $B(0,1)\times[-\nicefrac{c}{2},\nicefrac{c}{2}]$
              $(0<c<1)$ such that $(0,0)\in \B(x,t);$

    \item[{\bf A2.}]
	   For all $s\in\pi_2(\B(x,t)),$  the set $\B^s(x,t):=\left\{y\in\Rn\colon (y,s)\in\B(x,t)\right\}$  is symmetric with respect to the origin;

     \item[{\bf A3.}]Continuity respect to $(x,t):$ Given $(x,t)\in\Omega_T,$ if $\{(x_n,t_n)\}_{n\in\mathbb{N}}\subset\Omega_T$ and
    $(x_n,t_n)\to(x,t)$ as $n\to\infty$ then for every $(y,s)\in\B(x,t)$ there exist $(y_n,s_n)\in\B(x_n,t_n)$ such
    that $(y_n,s_n)\to (y,s)$ as $n\to\infty.$ Moreover if  $(y_n,s_n)\in\B(x_n,t_n)$  and $(y_n,s_n)\to (y,s)$ as $n\to\infty,$ then
    $(y,s)\in\B(x,t);$

    \item[{\bf A4.}] For every $v\in\Rn\setminus\{0\},$ there exist a unique $(z,r)\in\B(x,t)$ such that
    \[
    \min\left\{ \langle v,y\rangle \colon y\in\pi_1(\B(x,t)) \right\}=
    \langle v,z\rangle.
    \]
    From now on, $J_{(x,t)}(v)$  and $I_{(x,t)}(v)$ denote the point $z$ and
    the time $r$ respectively. Observe that
    \[
    \langle v, J_{(x,t)}(v)\rangle \neq 0,
    \]
    and
    \[
    (J_{(x,t)}(\lambda v), I_{(x,t)}(\lambda v))=(J_{(x,t)}(v),I_{(x,t)}(v))
    \]
    for any  $\lambda>0.$
    Therefore, $(J_{(x,t)}(v),I_{(x,t)}(v))$ depends only in the direction of $v.$
    Moreover,  $(-J_{(x,t)}(v),I_{(x,t)}(v))\in\B(x,t)$
     \[
    \max\left\{ \langle v,y\rangle \colon y\in\pi_1(\B(x,t)) \right\}=\langle v,
    -J_{(x,t)}(v)\rangle.
    \]
    In addition, we require that,
    \[
    J_{(x,t)}\colon\partial B(0,1)\to \partial\pi_1(\B(x,t))
    \]
    is surjective.
%
  \end{enumerate}

  \begin{exa}\label{ejemplo} We now give some examples of possible choices of sets $\B(x,t)$.
\begin{enumerate}
	\item For any $(x,t)\in\Omega_T,$ we define 
		\[
		\B_1(x,t):=\left\{(y,s)\in B(0,1)\times\left[-\frac{c}2, \frac{c}2\right]\colon |y|^2+|s|^2\le\rho^2\right\}
		\]
		where $0<\rho<\min\{1,\nicefrac{c}2\}.$ For this family of sets, 
		\[
		(J_{(x,t)}(v),I_{(x,t)}(v))=\left(-\frac{\rho v}{|v|},0\right)
		\]
	for all $(x,t)\in\Omega_T$ and $v\in\Rn\setminus\{0\}.$
	\item For any $(x,t)\in\Omega_T,$ we define  
		\[
		\B_2(x,t):=\left\{(y,s)\in B(0,1)\times\left[0,\frac{c}2\right]\colon |y|^2\le\frac{2\rho s}{c}\right\}
		\]
		where $0<\rho<1.$ Then
		\[
		(J_{(x,t)}(v),I_{(x,t)}(v))=\left(-\frac{\rho v}{|v|},\frac{c}2\right)
		\]
	for all $(x,t)\in\Omega_T$ and $v\in\Rn\setminus\{0\}.$

\end{enumerate}
	\end{exa}

Then, we define the set of possible movements for any $(x,t)\in\Omega_T.$
Given $(x,t)\in\Omega_T$ and $\ep>0$ small, the set of possible movements in
$(x,t)$ is given by a scaled version of the original family of sets. We let
\begin{equation}
  \Ae(x,t):=\left\{ (x,t)+\left(\ep y,\ep^2\frac{1-c}{c}s-\ep^2\frac{c+1}{2}\right)\colon (y,s)\in\B(x,t) \right\}.
  \label{ae}
\end{equation}

In the rest of this section, we only assume that the family $\{\B(x,t)\}_{(x,t)\in\Omega_T}$ has property {\bf A1}, the rest of properties
will be used in the following sections.

\begin{re}\label{tiempofinito}
Since, by assumption {\bf A1} the time $t_k$ decreases at least $c\ep^2$ at each round of the game, given $(x,t)\in\Omega_T$
we have that
$
0\le\tau(x,t)<\frac{t}{c\ep^2}+1.
$
Then,  the number of rounds that the player need to end the game when starting from $(x_0,t_0)$  is finite.
\end{re}

We have a Dynamic Programming Principle for our game. For the proof see \cite[Chapter 3]{msb}.

\begin{lem}[DPP]\label{DPP}
The $\ep-$value of the game for the Player I satisfies
\[
\begin{cases}
  u_I^\ep(x,t)={\dis \frac12\sup_{(y,s)\in\Ae(x,t)} u_{I}^\ep\left(y,s\right)+
  \frac12\inf_{(y,s)\in\Ae(x,t)} u_{I}^\ep\left(y,s\right)} &\mbox{if }(x,t)\in\Omega_T,\\
  u_I^\ep(x,t)= F(x,t) & \mbox{if } (x,t)\in\Gamma_\ep.
\end{cases}
\]
The $\ep-$value function for the Player II, $u_{II}^\ep(x,t),$ satisfies the same equations.
\end{lem}

Our next goal is to state a comparison principle for the $\ep-$values functions and then we will show that the game has a value.

\begin{de}
  A function $v$ is a subsolution of DPP if
 \[
  v(x,t)\ge{\dis \frac12\sup_{(y,s)\in\Ae(x,t)} v\left(y,s\right)+
\frac12  \inf_{(y,s)\in\Ae(x,t)} v\left(y,s\right)}
\]
for $(x,t)\in\Omega_T.$

Respectively, the supersolutions are defined by reversing the inequality for $v,$ that is
 \[
  v(x,t)\le{\dis \frac12\sup_{(y,s)\in\Ae(x,t)} v\left(y,s\right)+
  \frac12 \inf_{(y,s)\in\Ae(x,t)} v\left(y,s\right)}
\]
for $(x,t)\in\Omega_T.$
\end{de}

\begin{te}[Comparison Principle] \label{comparison}
 Let $\Omega\subset\Rn$ be a bounded open set and  $v$ be a subsolution (supersolution) of DPP and $v\le F$ ($v\ge F$) in $\Gamma_\ep$
 we have that $u_{II}^\ep\ge v$ ($u_{I}^\ep\le v$) in $\Omega_T.$
\end{te}

Hence, we have that
$u_I^\ep$ ($u_{II}^\ep$) is the lowest (largest) function that satisfies the DPP with boundary values $F.$

\begin{te}\label{valor}
   Let $\Omega\subset\Rn$ be a bounded open set and $F$ a given payoff function in $\Gamma_\eta.$ Then the game has a $\ep-$value, i.e., $
   u_I^\ep=u_{II}^\ep$.
\end{te}

The proofs of above theorems are analogous to the proofs of Theorem 4.4 and Theorem 4.5  in \cite{mpr-parabolic}, respectively.

Observe that, using Theorems \ref{comparison} and \ref{valor}, we have that there exists a unique function that verify the DPP with a fixed
boundary datum.

\begin{te}\label{eu}
 Let $\Omega\subset\Rn$ be a bounded smooth open set and $F$ a given payoff function in $\Gamma_\eta.$ There exists a unique function
  $u_\ep$ in $\Omega_T$ that verify the DPP with boundary values $F.$ Moreover, the function $u_\ep$ coincides with the $\ep-$value of the game.
\end{te}

Theorems \ref{comparison} and \ref{eu} imply the comparison principle for functions that verify the DPP.

\begin{te}\label{comparison2}
 Let $\Omega\subset\Rn$ be a bounded smooth open set and  $v,u$ be  functions verifying the DPP with boundary values $H$ and $F$ in
 $\Gamma_\ep$ respectively. Then, if $H\le F,$ we have that $v\le u$ in $\Omega_T.$
\end{te}

As a consequence, we get that solutions to the DPP are uniformly bounded for bounded $\ep.$

\begin{co}\label{cotunif1}
 Let $\Omega\subset\Rn$ be a bounded smooth open set and  $u$ be a function verifying the DPP with boundary values $F$ in
 $\Gamma_\ep.$ Then,
 \[
 \inf_{\Gamma_\ep}F\le u(x,t)\le\sup_{\Gamma_\ep}F
 \]
 for any $(x,t)\in\Omega_T.$
\end{co}


\section{Uniform convergence}\label{uniconv}

In this section, we prove that, extracting a subsequence if necessary, we have uniform convergence of $u_\ep$ as $\ep\to0.$ To this end, we adapt
some ideas from \cite{gr} and we use the following modification of Arzela--Ascoli lemma, see \cite{mpr-definition} for the proof.

\begin{lem}\label{AA}
  Let $\{f_\ep:\overline{\Omega_T}\to\mathbb{R}, \ep>0\}$ be a set of functions such that:
  \begin{enumerate}
    \item There exist a positive constant $C$ so that $|f_\ep(x,t)|<C$ for every $\ep>0$ and every $(x,t)\in\overline{\Omega_T};$
    \item Given $\nu>0,$ there exist positive constants $r_0$ and $\ep_0$ such that for any $\ep<\ep_0$ and any
      $(x,t),(y,s)\in\overline{\Omega_T}$ with $|x-y|+|t-s|<r_0,$ it holds that
$      |f_{\ep}(x,t)-f_\ep(y,s)|<\nu$.
  \end{enumerate}
  Then there exists a uniformly continuous function $f:\overline{\Omega_T}\to\R$ and a subsequence still denoted by $\{f_\ep\}_{\ep>0}$ such that
  \[
  f_\ep\to f \quad \mbox{uniformly in } \overline{\Omega_T}
  \]
  as $\ep\to0.$
\end{lem}

Now, let $\eta>0$ and $F:\Gamma_\eta\to\R$ be a bounded Borel function,
we consider the family of functions  $\{u_\ep\}_{\ep>0}$ where $u_\ep$ are  the $\ep-$values
of the game with payoff function $F$ for each $\ep>0.$ Observe that, by Corollary \ref{cotunif1}, we have that
\begin{equation}
  \label{cotunif2} |u_\ep (x,t)|\le \sup_{(w,s)\in\Gamma_\eta}|F(w,s)| \quad \forall (z,s)\in\overline{\Omega_T}.
\end{equation}
Therefore, the family $\{u_\ep\}_{\ep>0}$ satisfies the first condition in Lemma \ref{AA}. Then, to prove the uniform convergence of
the family, we only need to show that $\{u_\ep\}_{\ep>0}$ satisfies also the second condition of the lemma. To prove this, we need to use properties {\bf A1}-{\bf A2} for the family of sets $\{\B(x,t)\}_{(x,t)\in\Omega_T}$.

\begin{re}\label{slineal}Observe that,
  by {\bf A2}, any linear function $l(x)=\langle v,x\rangle + b$ is a solution of the DDP (with $F(x,t)=l(x)$ in $\Gamma_\ep$),
  where $v\in\R^N$ and $b\in\R.$
\end{re}

\begin{lem}\label{cec} Let $\Omega$ be bounded convex domain with $\partial \Omega\in C^2$ and positive curvature, $f:\Omega_\eta\to\R$ be a Lipschitz
  continuous function and assume that the family of sets $\{\B(x,t)\}_{(x,t)\in\Omega_T}$ satisfies the properties {\bf A1}--{\bf A2}. Then,
  if we take $F(x,t)=f(x)$ as our payoff function,
  given $\nu>0$ there exist positive constants $r_0$ and $\ep_0$ such that for any $\ep<\ep_0$ and any
      $(x,t),(y,s)\in\overline{\Omega_T}$ with $|x-y|+|t-s|<r_0,$ it holds that $
      |u_{\ep}(x,t)-u_\ep(y,s)|<\nu$,
      where $u_\ep$ is the $\ep-$value of the game with boundary value $F(x,t),$
      i.e. $\{u_\ep\}_{\ep>0}$ satisfies the second condition of the Lemma \ref{AA}.
\end{lem}

\begin{proof} We divide the proof in four cases.

  {\bf Case 1}. The case $(x,t),(y,s)\in \Gamma$ is a consequence of the fact that $f$ is assumed to be Lipschitz.

  {\bf Case 2}. Now, we study the case $(x,t)\in\Omega_T$ and $(y,t)\in\Gamma_\varepsilon$ with $y\in\Theta_\varepsilon$.
As in the proof of  \cite[Lemma 14]{gr}, using that $\partial\Omega\in C^2,$ we can choose an hyperplane $\Pi_0$ such that
$\Pi_0$ is tangent to $\Omega$ at some point $y_0\in \partial \Omega$ and $y$ lies in the
outward normal direction to $\partial\Omega$ at $y_0.$ Via a translation and rotation of the coordinate axes,
we can assume that $y_0=0$ and $\Pi_0 = \{x^N=0\}.$ Moreover, using that $\partial \Omega$ has positive curvature, there exist a positive
constants $k$ and $K$ such that for $U= B(0,k)\times\{x\in\Rn\colon -k<x^N< k\}$  we have that
\[
\Omega\cap U \subset \left\{ x\in \mathbb{R}^N\colon x^N\le-K\sum_{i=1}^{N-1}(x^i)^2 \right\}.
\]

On the other hand, by the definition of  $\Theta_\varepsilon,$ if $x\in\Theta_\varepsilon\cap U$ there exists
$z\in\partial\Omega\cap U$ such that $|x-z|\le \varepsilon.$ Then, for any $\delta>0$ if  $-8\delta<z^N< 8\delta$ and
$0<\varepsilon< \left(\frac{8\delta}{K}\right)^{1/2}$
we have that
\[
K\sum_{i=1}^{N-1}(x^i)^2\le 24\delta+K\varepsilon^2<32\delta.
\]
Then, for any $0<\delta<\frac{1}{2K}$ and $0<\varepsilon<\left(\frac{8\delta}{K}\right)^{1/2}$
\[
\mathcal{C}_{\delta,\varepsilon}\subset\left\{ w\in \R^{N-1}\colon K\sum_{i=1}^{N-1}(w^i)^2<32\delta\right\}\times(-8\delta,8\delta)\subset
B(0,\rho_{\delta})
\]
where $\mathcal{C}_{\delta,\varepsilon}:=\Theta_\varepsilon\cap\{x\in\Rn\colon -8\delta<x^{N}<8\delta\}$ and
$\rho_\delta=\left(\frac{64\delta}{K}\right)^{1/2}.$ Therefore,
\begin{equation*}
  f(x)\le \alpha:=\sup_{z\in\Theta_\varepsilon\cap B(0,\rho_\delta)}f(z)\quad \forall x\in \mathcal{C}_{\delta,\varepsilon}.
\end{equation*}

Now, we consider the function $v:\Rn\times\R\to\R$
\[
v(x,t)=ax^N+b
\]
where $a$, $b$ are given by
\[
a=-\frac{\beta-\alpha}{4\delta+\varepsilon},\qquad b=\frac{4\delta\alpha+\varepsilon\beta}{4\delta+\varepsilon}
\]
with
\[
\beta=\sup_{z\in\Omega_\eta}f(z).
\]
Observe that $v$ is decreasing with respect to the space variable $x^N,$  $v(x,t)\equiv\beta$ on $\{x\in\Rn\colon x^N=-4\delta\},$
$v(x,t)\equiv\alpha$ on $\{x\in\Rn\colon x^N=\ep\}$ and, by Remark~\ref{slineal}, $v$ is a solution of the DPP.

If, $a=0$ then we have $\alpha=\beta$ then, by Corollary \ref{cotunif1},
\begin{equation*}
  u_\varepsilon(x,t)\le\alpha \quad\forall(x,t)\in\Omega_T\cup\Gamma_\varepsilon.
\end{equation*}

Now, we consider the case $a\neq0.$  We observe that, if we take
\[
\Omega'=\Omega\cap\{x\in\Rn\colon -4\delta<x^N<0\},
\quad \Omega'_\varepsilon=\{x\in\Rn\colon\mbox{dis}(x,\Omega')\le\varepsilon\},
\]
\[
\Gamma'_\varepsilon=\left(\Omega'_\varepsilon\times(-\varepsilon^2,0]\right)
\cup\left((\Omega'_\varepsilon\setminus\Omega')\times(0,T]\right),
\]
we have that $v$ and $u_\varepsilon$ are solution of DPP in $\Omega'_T=\Omega\times(0,T]$ with payoff functions $v$ and $u_\ep$ in
$\Gamma'_\varepsilon$ respectively.
Since, by definition of $v,$ $u_\varepsilon(x,t)\le v(x,t)$ in $\Gamma'_\varepsilon,$ using Theorem \ref{comparison2}, we have that
\begin{equation*}
  u_\varepsilon(x,t)\le v(x,t) \mbox{ in } \Omega'_\varepsilon\times(-\varepsilon^2,T].
\end{equation*}

On the other hand, there exists $\varepsilon_1>0$ (depending of
$\delta,$ $\alpha$ and $\beta$) such that
\[
v(x,t)\le \alpha+\frac12(\beta-\alpha)
\]
in $(\Omega\cap\{x\in\Rn\colon-\delta-\varepsilon<x^N<
\varepsilon\})\times(-\varepsilon^2,T]$ for all $\varepsilon<
\varepsilon_1$. Then, by an iterative process, we have that for any
$m\in\mathbb{N}$ there exists $\varepsilon_m>0$ (depending of
$\delta,$ $\alpha$ and $\beta$) such that
\[
u_\varepsilon(x,t)\le \alpha+\left(\frac12\right)^m(\beta-\alpha)
\]
in $(\Omega\cap\{x\in\Rn\colon-\delta/4^m-\varepsilon<x^N<
\varepsilon\})\times(-\varepsilon^2,T],$ for all $\varepsilon<\varepsilon_m$.

The argument needed to obtain an analogous lower bound is similar.

On the other hand, since $f$ is  Lipschitz,  we have
\[
|\alpha-f(y)|\le C \delta^{1/2} \quad \forall y\in B(0,\delta).
\]
Therefore, given $\nu>0$ we can choose small $\delta,\varepsilon>0$ and large enough $m\in\mathbb{N}$ such that $x\in\Omega$ and
$y\in\Theta_\varepsilon$ with $|x-y|<\delta/4^m$ it holds
\begin{equation}\label{pdes}
|u_\varepsilon(x,t)-F(y,s)|=
|u_\varepsilon(x,t)-f(y)|<\nu\quad\forall t\in(0,T]\,\forall s\in(-\varepsilon^2,T].
\end{equation}

{\bf Case 3}. The case $(x,t)\in\Omega_T$ and $(y,s)\in\Omega\times(-\varepsilon^2,0].$
First, we assume that $x=y$ and Player I follows a strategy $S_I^x$ where he points to $y$ and Player II follows any
strategy. Then
\[
M_k=|x_k-x|^2-k\varepsilon^2
\]
is a supermartingale. Indeed,
\[
\Esy[|x_k-x|^2|x,x_1,\dots,x_{k-1}]\le|x_{k-1}-x|^2+\varepsilon^2.
\]
Then, by the optimal stopping theorem and Remark \ref{tiempofinito}, we have
\[
\Esy[|x_\tau-x|^2]\le C(t+\varepsilon^2)
\]
where $C$ is a constant independent of $x$ and $\varepsilon.$ Thus, by Jensen's inequality, we get
\[
\Esy[|x_\tau-x|]\le C(t+\varepsilon^2)^{\nicefrac12}\le C(t^{\nicefrac12}+\varepsilon).
\]
Hence,
\[
  F(x,s)-L C(t^{\nicefrac12}+\varepsilon)\le \Esy [F(x_\tau,t_{\tau})]
  \le  F(x,s)+L C(t^{\nicefrac12}+\varepsilon)
\]
where $L$ is the Lipschitz constant of $f.$
Then,
\begin{align*}
  u_\varepsilon(x,t)&=\sup_{S_I}\inf_{S_{II}} \mathbb{E}_{S_{I},S_{II}}^{(x,t)}[F(x_\tau,t_\tau)]\\
  &\ge\inf_{S_{II}}\Esy[F(x_\tau,t_\tau)]\\
  &\ge F(x,s)-LC(t^{1/2}+\varepsilon).
\end{align*}
Therefore
\[
  u_\varepsilon(x,t)-F(x,s)= u_\varepsilon(x,t)-f(x)\ge-C(t^{1/2}+\varepsilon).
\]
Similarly, choosing for Player II the strategy where he points to $x,$ we have that
\[
 u_\varepsilon(x,t)-f(x)\le C(t^{1/2}+\varepsilon).
\]
We can conclude that
\[
| u_\varepsilon(x,t)-f(x)|\le C(t^{1/2}+\varepsilon).
\]
Finally, if $x\neq y,$ we utilize the above inequality and we have that
\begin{align*}
 |u_\varepsilon(x,t)-u_\varepsilon(y,s)|&= |u_\varepsilon(x,t)-f(y)|\\
 &\le |u_\varepsilon(x,t)-f(x)| + |f(x)-f(y)|\\
  &\le C(|x-y|+t^{\nicefrac12}+\varepsilon).
\end{align*}
Therefore, by \eqref{pdes} and the above inequality, given $\nu>0,$ there exist $\varepsilon_0,r_0>0$ so that
\begin{equation}\label{sdes}
 |u_\varepsilon(x,t)-u_\varepsilon(y,s)|\le \nu
\end{equation}
for all $\varepsilon<\varepsilon_0$ and for any $(x,t)\in\Omega_T$ and $(y,s)\in\Gamma_\varepsilon$ such that
\mbox{$|x-y|+|t-s|<r_0.$}

{\bf Case 4}. Finally, we study the case $(x,t),(y,s)\in\Omega_T.$
We consider, as in the proof of \cite[Lemma 17]{mpr-parabolic},
\[
\hat{\Omega}_T=\left\{(z,t)\in\Omega_T\colon d((z,t),\Gamma)>\frac{r_0}{3}\right\}.
\]
where
\[
d( (z,t),\Gamma)=\inf\{|z-y|+|t-s|\colon (y,s)\in\Gamma\},
\]
and the boundary strip
\[
\hat{\Gamma}=\left\{(x,t)\in\overline{\Omega_T}\colon d((z,t),\Gamma)\le\frac{r_0}{3}\right\}.
\]
Let $(x,t),(y,s)\in\Omega_T$ such that $|x-y|+|t-s|<\frac{r_0}{3}.$
First, if $(x,t),(y,s)\in\hat{\Gamma},$ by comparison the values $(x,t)$ and $(y,s)$ to the nearby boundary values and using \eqref{sdes}, we
have
$|u_\varepsilon(x,t)-u_\varepsilon(y,s)|\le \nu$
for all $\varepsilon<\varepsilon_0.$

Finally, the case $(x,t),(y,s)\in\hat{\Omega}_T.$ Without loss of generality, we can assume that $t>s.$ Define
\[
\hat{F}(z,h)=u_\varepsilon(z-x+y,h-t+s)+3\nu,\quad \mbox{for } (z,h)\in\hat{\Gamma}.
\]
Then, by the reasoning above,
\[
\hat{F}(z,h)\ge u_\varepsilon(z,h)     \quad \forall (z,h)\in\hat{\Gamma}.
\]
Let $\hat{u}_\varepsilon$ be a solution of DPP in $\hat{\Omega}_T$ with the boundary values $\hat{F}$ in $\hat{\Gamma}.$ By comparison
principle and uniqueness, we have
\[
u_\varepsilon(x,t)\le\hat{u}_\varepsilon(x,t)=u_\varepsilon(y,s)+3\nu.
\]
The reverse bound follows by a similar argument.
\end{proof}

Now, by \eqref{cotunif2} and using Lemma \ref{AA} and Lemma \ref{cec}, we  get the main result of this section.

\begin{te}\label{conv}
  Under the same hypothesis in Lemma \ref{cec}. Let $\{u_\varepsilon\}_{\varepsilon>0}$ be the family of solution of DPP in $\Omega_T$ with a  fixed
  Lipschitz continuous datum $F(x,t)=f(x)$ in $\Gamma.$ Then, there exists a subsequence still denoted by $\{u_\varepsilon\}_{\varepsilon>0}$
  and a uniformly continuous function $u$ such that
  \[
  u_{\varepsilon}\to u \quad\mbox{uniformly in }\overline{\Omega}_T
  \]
  as $\varepsilon\to0^+.$
\end{te}

\section{The limit equation} \label{sect.eq.limite}

Throughout this section, $\Omega$ is bounded convex domain with $\partial \Omega\in C^2$ and positive curvature, $f:\Omega_\eta\to\R$ is a Lipschitz
continuous function and we take $F(x,t)=f(x)$ as our payoff function. We assume that the family of sets
$\{\B(x,t)\}_{(x,t)\in\Omega_T}$ satisfies the full set of properties {\bf A1}--{\bf A4}.

The aim of this section is to prove that the function $u,$ given by Theorem \ref{conv}, is a viscosity solution of the following PDE
\begin{equation}
  \begin{cases}
  G(D^2 u(x,t), \nabla u(x,t), u_t(x,t),x,t)=0 &\mbox{in } \Omega_T,\\
  u(x,t)=F(x,t)&\mbox{in }\Gamma,
\end{cases}
  \label{ve}
\end{equation}
where $D^2u$ is the Hessian matrix of $u$
and $G:\SN\times\Rn\times\R\times\Omega_T\to\Rn$ is defined by
\[
G(M,v,s,x,t)= \begin{cases}
  \left(-\frac{1-c}{c} I_{(x,t)}(v) +\frac{c+1}{2}\right)s-\frac12 \langle MJ_{(x,t)}(v),J_{(x,t)}(v)\rangle &\mbox{if } v\neq0,\\
  \left(-\frac{1-c}{c} \hat{I}_{(x,t)}(s) +\frac{c+1}{2}\right) s &\mbox{if } v=0,
\end{cases}
\]
where  $J_{(x,t)}$, $I_{(x,t)}$ are defined in Section \ref{dotg} and  $\hat{I}_{(x,t)}(s)$ is defined as the unique time such that
\[
\begin{cases}
(0,\hat{I}_{(x,t)}(s))\in\B(x,t),\\
\hat{I}_{(x,t)}(s)s=\min\{rs\colon (0,r)\in\B(x,t)\}
\end{cases}
\]
if $s\in\R\setminus\{0\}$ and $\hat{I}_{(x,t)}(0):=0.$

First, we will give the precise definition of viscosity solution to \eqref{ve} following \cite{IK}.
We denote by $G^*$ and $G_*$ the upper and lower semicontinuous envelopes of $G$ respectively, i.e.
\[
G^*(M,v,s,x,t):=\limsup_{\varepsilon\to0}\left\{ G(\hat{M},\hat{v},\hat{s},\hat{x},\hat{t})\colon (\hat{M},\hat{v},\hat{s},\hat{x},\hat{t})\in
\mathbb{C}_{\varepsilon}(M,v,s,x,t)\right\}
\]
where
\[
\mathbb{C}_\varepsilon(M,v,s,x,t):=\left\{\|M-\hat{M}\|+
|s-\hat{s}|+|v-\hat{v}|+|x-\hat{x}|+|t-\hat{t}|<\varepsilon \right\}
\]
and
\[
G_*(M,v,s,x,t):=-(-G)^*(M,v,s,x,t)
\]
for every $(M,v,s,x,t)\in\SN\times\Rn\times\R\times\Omega_T.$

\begin{de}\label{vsol}
  A function $u\in C(\overline{\Omega_T})$ is a viscosity solution to \eqref{ve} if $u(x,t)=F(x,t)$ on $\Gamma$ and the following two
  conditions hold:
  \begin{enumerate}[(i)]
    \item For every $\phi\in C^{2,1}(\overline{\Omega_T})$ such that $u-\phi$ has a strict minimum at $(x_0,t_0)\in\Omega_T$ we have
      \[
      G^*(D^2\phi(x_0,t_0),\nabla\phi(x_0,t_0),\phi_t(x_0,t_0),x_0,t_0)\ge0;
      \]
 \item For every $\phi\in C^{2,1}(\overline{\Omega_T})$ such that $u-\phi$ has a strict maximum at $(x_0,t_0)\in\Omega_T$ we have
      \[
      G_*(D^2\phi(x_0,t_0),\nabla\phi(x_0,t_0),\phi_t(x_0,t_0),x_0,t_0)\le0.
      \]
  \end{enumerate}
\end{de}

Now we characterize the upper and lower envelopes for the function $G.$

\begin{lem}
  For any $(M,v,s,x,t)\in\SN\times\Rn\times\R\times\Omega_T,$ we have
  \[
  G^*(M,v,s,x,t)=
  \begin{cases}
    G(M,v,s,x,t)&\mbox{if }v\neq0,\\
   \displaystyle  \max_{(z,r)\in \B(x,t)}\left\{\left(-\frac{1-c}{c} r +\frac{c+1}{2}\right)s-\frac12\langle M z,z\rangle\right\} &\mbox{if }v=0,
  \end{cases}
  \]
and
  \[
  G_*(M,v,s,x,t)=
  \begin{cases}
    G(M,v,s,x,t)&\mbox{if }v\neq0,\\
    \displaystyle \min_{(z,r)\in \B(x,t)}\left\{\left(-\frac{1-c}{c} r +\frac{c+1}{2}\right)s-\frac12\langle M z,z\rangle\right\} &\mbox{if }v=0.
  \end{cases}
  \]
\end{lem}

\begin{proof}
  We only prove the characterization for $G^*,$ the proof for $G_*$ is similar.

  {\bf Step 1}. First we prove,
  \[
  G^*(M,v,s,x,t)=G(M,v,s,x,t)
  \]
  if $v\neq0.$
  Let $(M_n,v_n,s_n,x_n,t_n)\to(M,v,s,x,t).$ As $v\neq0,$ and we can assume
  that $v_n\neq0.$ Then, by definition, for any $n\in \mathbb{N},$
  \begin{align*}
   G(M_n,v_n,s_n,x_n,t_n)=&
  \left(-\frac{1-c}{c} I_{(x_n,t_n)}(v_n) +\frac{c+1}{2}\right)s_n\\
  -&\frac12 \langle M_nJ_{(x_n,t_n)}(v_n),J_{(x_n,t_n)}(v_n)\rangle.
  \end{align*}
  Since $(J_{(x_n,t_n)}(v_n),I_{(x_n,t_n)}(v_n))\in \B(x_n,t_n)\subset
  B(0,1)\times[-\nicefrac{c}{2},\nicefrac{c}{2}]$ for every $n\in\mathbb{N},$
  there exists a subsequence still denote by
  $\{(J_{(x_n,t_n)}(v_n),I_{(x_n,t_n)}(v_n))\}_{n\in\mathbb{N}}$
  and $(y,r)\in B(0,1)\times[-\nicefrac{c}2,-\nicefrac{c}2]$ such that
  \[
  (J_{(x_n,t_n)}(v_n),I_{(x_n,t_n)}(v_n))\to(y,r)
  \mbox{ as } n\to+\infty.
  \]
  Moreover, by {\bf A3}, $(y,r)\in\B(x,t).$
  Then, by definition of $(J_{(x,t)}(v),I_{(x,t)}(v)),$ we have that
  \begin{equation}\label{desigs}
  \langle v,y\rangle\ge
  \langle v, J_{(x,t)}(v)\rangle.
  \end{equation}
  On the other hand, by {\bf A3}, there exist $(y_n,r_n)\in\B(x_n,t_n)$ such that
  \[
  (y_n,r_n)\to(J_{(x,t)}(v),I_{(x,t)}(v)) \mbox{ as } n\to+\infty.
  \]
  Thus, by definition of $(J_{(x_n,t_n)}(v_n),I_{(x_n,t_n)}(v_n)),$ we have
 \[
 \langle v_n, y_n\rangle\ge
 \langle v_n,J_{(x_n,t_n)}(v_n)\rangle \quad \forall n\in\mathbb{N}.
  \]
  Then, taking limit as $n\to +\infty$ and using \eqref{desigs}, we get
 \[
  \langle v,y\rangle=
 \langle v,J_{(x,t)}(v)\rangle.
 \]
 Thus, by {\bf A4}, $y=J_{(x,t)}(v)$ and $r=I_{(x,t)}(s),$ then we
 have
 \[
 G(M_n,v_n,s_n,x_n,t_n)\to G(M,v,s,x,t)
 \]
  as $n\to+\infty$ and therefore
 \[
 G^*(M,v,s,x,t)=G(M,v,s,x,t)
 \]
if $v\neq0.$

 {\bf Step 2}. Now, we consider the case $v=0$ and we show that
 \[
 G^*(M,0,s,x,t)\le \max_{(z,r)\in\B(x,t)}\left\{\left(-\frac{1-c}{c} r +\frac{c+1}{2}\right)s-\frac12\langle M z,z\rangle\right\} .
 \]
 Let $(M_n,v_n,s_n,x_n,t_n)\to(M,0,s,x,t).$
 If $v_n=0$ for $n$ large,
 \[
 G(M_n,v_n, s_n, x_n, t_n)=\left(-\frac{1-c}{c} \hat{I}_{(x_n,t_n)}(s_n) +\frac{c+1}{2}\right)s_n
 \]
 Then, as in step 1, extracting a subsequence still denoted
 $\{(M_n,v_n,s_n,x_n,t_n)\}_{n\in\mathbb{N}},$ we have that
 \[
  (0,\hat{I}_{(x_n,t_n)}(s_n))\to (0,r_0)\in\B(x,t),
  \]
and therefore
 \begin{equation}
	  G(M_n,v_n,s_n,x_n,t_n)\to \left(-\frac{1-c}{c} r_0 +\frac{c+1}{2}\right)s
	 \label{2d1}
 \end{equation}
  If $v_n\neq0$ for $n$ large
  \begin{align*}
   G(M_n,v_n,s_n,x_n,t_n)=&
  \left(-\frac{1-c}{c} I_{(x_n,t_n)}(v_n) +\frac{c+1}{2}\right) s_n \\
  -&\frac12 \langle M_nJ_{(x_n,t_n)}(v_n),J_{(x_n,t_n)}(v_n)\rangle.
  \end{align*}
  Then, arguing again as in step 1, extracting a subsequence that we still denote by
  $\{(M_n,v_n,s_n,x_n,t_n)\}_{n\in\mathbb{N}},$ we have that
  there exist $(w,h)\in\B(x,t)$ such that
  \begin{equation}\label{2d2}
  G(M_n,v_n,s_n,x_n,t_n)\to \left(-\frac{1-c}{c} h +\frac{c+1}{2}\right) s
  -\frac12 \langle Mz,z\rangle.
  \end{equation}
  as $n\to+\infty.$ Thus, by \eqref{2d1} and \eqref{2d2},
  \[
   G^*(M,0,s,x,t)\le  \max_{\B(x,t)}\left\{\left(-\frac{1-c}{c} r +\frac{c+1}{2}\right)s-\frac12\langle M z,z\rangle\right\}.
   \]

  {\bf Step 3.} Finally, we prove that,
\[
 \max_{(z,r)\in\B(x,t)}\left\{\left(-\frac{1-c}{c} r +\frac{c+1}{2}\right)s-\frac12\langle M z,z\rangle\right\} \le G^*(M,0,s,x,t) .
\]
Since $\B(x,t)$ is a compact, there exists $(Z,R)$ such that
  \[
     \max_{(z,r)\in\B(x,t)}\left\{\left(-\frac{1-c}{c} s +\frac{c+1}{2}\right)r-\frac12\langle M z,z\rangle\right\}=
     \left(-\frac{1-c}{c} s +\frac{c+1}{2}\right)R-\frac12\langle M Z,Z\rangle.
  \]
First, we suppose that $Z=0.$ If $s\neq0,$ we have that $\hat{I}_{(x,t)}(s)=R$ and then
 \begin{align*}
  G(M,0,s,x,t)&= \left(-\frac{1-c}{c} R +\frac{c+1}{2}\right)s\\
  &=\max_{\B(x,t)}\left\{\left(-\frac{1-c}{c} s +\frac{c+1}{2}\right)r-\frac12\langle M z,z\rangle\right\}.
 \end{align*}
 If $s=0$ then
\[
\max_{(z,r)\in\B(x,t)}\left\{\left(-\frac{1-c}{c} r +\frac{c+1}{2}
\right)s-\frac12\langle M z,z\rangle\right\}=0=G(M,0,0,x,t).
 \]
Therefore, if $Z=0$
\[
 \max_{(z,r)\in\B(x,t)}\left\{\left(-\frac{1-c}{c} r +\frac{c+1}{2}
 \right)s-\frac12\langle M z,z\rangle\right\} \le G^*(M,0,s,x,t) .
\]
If $Z\neq0,$ without loss generality, we can assume that $(Z,R)\in\partial\B(x,t).$
 By {\bf A4}, $J_{(x,t)}:\partial B(0,1)\to\partial\pi_1(\B(x,t))$ is subjective, then there exists $v\in B(0,1)$ such that
 $J_{(x,t)}(v)=Z.$ Thus, using again {\bf A4},
 \[
 \left(J_{(x,t)}\left(\frac{v}{n}\right),I_{(x,t)}\left(\frac{v}{n}\right)\right)=(J_{(x,t)}(v),I_{(x,t)}(v))=(Z,R)
 \]
 for all $n\in\mathbb{N}$
 \begin{align*}	
 G\left(M,\frac{v}{n},s,x,t\right)=&
 \left(-\frac{1-c}{c} I_{(x,t)}\left(\frac{v}{n}\right) +\frac{c+1}{2}\right)s-\frac12 \langle MJ_{(x,t)}\left(\frac{v}{n}\right),J_{(x,t)}\left(\frac{v}{n}\right)\rangle\\
 =&\left(-\frac{1-c}{c} I_{(x,t)}(v) +\frac{c+1}{2}\right)s-\frac12 \langle MJ_{(x,t)}(v),J_{(x,t)}(v)\rangle\\
 =&\left(-\frac{1-c}{c} R +\frac{c+1}{2}\right)s-\frac12\langle M Z,Z\rangle
 \end{align*}
Hence
\[
 \max_{(z,r)\in\B(x,t)}\left\{\left(-\frac{1-c}{c} r +\frac{c+1}{2}\right)s-\frac12\langle M z,z\rangle\right\}\le G\left(M,0,s,x,t\right).
\]
The proof is now completed.
\end{proof}

\begin{te}
	If the values of the game $\{u_\varepsilon\}_{\varepsilon>0}$
	uniform converge to $u\in C(\overline{\Omega_T}),$ then $u$ is a
	viscosity solution to \eqref{ve} in the sense of Definition
	\ref{vsol}.
\end{te}

\begin{proof}
We begin by observing that, as $u_\varepsilon\rightrightarrows u$ and
$u_\varepsilon=F$ on $\partial\Omega_T,$ we have that $u=F$ on
$\Gamma.$

Now, we prove that if $\phi\in C^{2,1}(\overline{\Omega_T})$ and $u-\phi$ has a strict local minimum at $(x_0,t_0)$ then
\[
      G^*(D^2\phi(x_0,t_0),\nabla\phi(x_0,t_0),\phi_t(x_0,t_0),x_0,t_0)\ge0.
\]
As $u-\phi$ has a strict local minimum at $(x_0,t_0)$ we have that
\[
u(x,t)-\phi(x,t)> u(x_0,t_0)-\phi(x_0,t_0) \quad (x,t)\neq(x_0,t_0).
\]
Then, by the uniform convergence of $u_\ep$ to $u,$ there exists a sequence $(x_\ep,t_\ep)\to(x_0,t_0)$ such that
\begin{equation*}
u_\ep(x,t)-\phi(x,t)\ge u_\ep(x_\ep,t_\ep)-\phi(x_\ep,t_\ep)-o(\ep^2)
\end{equation*}
for every $(x,t)$ in a fixed neighborhood of $(x_0,t_0).$ Hence
\begin{align*}
	\max_{(y,s)\in\Ae(x_\ep,t_\ep)}u_\ep(y,s)\ge&\max_{(y,s)\in\Ae(x_\ep,t_\ep)}\phi(y,s) + u_\ep(x_\ep,t_\ep)-\phi(x_\ep,t_\ep)-o(\ep^2),\\
        \min_{(y,s)\in\Ae(x_\ep,t_\ep)}u_\ep(y,s)\ge&\min_{(y,s)\in\Ae(x_\ep,t_\ep)}\phi(y,s) + u_\ep(x_\ep,t_\ep)-\phi(x_\ep,t_\ep)-o(\ep^2).
\end{align*}
By Theorem \ref{eu}, we have that
\begin{align*}	 u_\ep(x_\ep,t_\ep)=&\frac12\left\{\max_{(y,s)\in\Ae(x_\ep,t_\ep)}u_\ep(y,s)+\min_{(y,s)
\in\Ae(x_\ep,t_\ep)}u_\ep(y,s)\right\}\\	\ge&\frac12\left\{\max_{(y,s)\in\Ae(x_\ep,t_\ep)}\phi(y,s)+\min_{(y,s)\in\Ae(x_\ep,t_\ep)}\phi(y,s)\right\}\\
	&+ u_\ep(x_\ep,t_\ep)-\phi(x_\ep,t_\ep)-o(\ep^2).
\end{align*}
Therefore
\begin{equation}\label{e1}
\phi_\ep(x_\ep,t_\ep)\ge\frac12\left\{\max_{(y,s)\in\Ae(x_\ep,t_\ep)}\phi(y,s)
+\min_{(y,s)\in\Ae(x_\ep,t_\ep)}\phi(y,s)\right\}-o(\ep^2).
\end{equation}
Now, let $(x_\ep^m,t_\ep^m)\in\Ae(x_\ep,t_\ep)$ such that
\begin{equation}\label{e2}
\min_{(y,s)\in\Ae(x_\ep,t_\ep)}\phi(y,s)=\phi(x_\ep^m,t_\ep^m)
\end{equation}
and let $\widetilde{x_\ep^m}$ by the symmetrical point of $x_\ep^m$ respect to $x_\ep,$ that is
$
\widetilde{x_\ep^m}=2x_\ep-x_\ep^m$.
Observe that,
\begin{equation}\label{igxm}
\widetilde{x_\ep^m}-x_\ep=x_\ep-x_\ep^m,
\end{equation}
and
by {\bf A2}, we have that
\begin{equation}\label{e3}
(\widetilde{x_\ep^m},t_\ep^m)\in\Ae(x_\ep,t_\ep).
\end{equation}
As $(x_\ep^m,t_\ep^m),(\widetilde{x_\ep^m},t_\ep^m)\in\Ae(x_\ep,t_\ep),$ by $\eqref{ae},$ there exists $(y_\ep^m,s_\ep^m)\in\B(x_\ep,t_\ep)$ such that
\begin{equation}\label{igxm1}
	\begin{aligned}
(x_\ep^m,t_\ep^m)&=(x_\ep,t_\ep)+ \left(\ep y_\ep^m,\ep^2\frac{1-c}{c}s_\ep^m-\ep^2\frac{c+1}2\right)\\
(\widetilde{x_\ep^m},t_\ep^m)&=(x_\ep,t_\ep)+ \left(-\ep y_\ep^m,\ep^2\frac{1-c}{c}s_\ep^m-\ep^2\frac{c+1}2\right)
\end{aligned}
\end{equation}
Then, using \eqref{e1},\eqref{e2} and \eqref{e3}, we have
\begin{equation*}
\phi(x_\ep,t_\ep)\ge\frac12\left\{\phi(\widetilde{x_\ep^m},t_\ep^m)+\phi(x_\ep^m,t_\ep^m)
\right\}-o(\ep^2).
\end{equation*}
Now, consider the Taylor expansion of second order of $\phi(\cdot,t_\ep^m)$ and using \eqref{igxm} and \eqref{igxm1}, we have that
\[
\phi(x_\ep,t_\ep)\ge \phi(x_\ep,t_\ep^m) + \frac{\ep^2}2 \langle D^2\phi(x_\ep,t_\ep^m) y_\ep^m, y_\ep^m\rangle + o(\ep^2).
\]
Then
\[
\frac{\phi(x_\ep,t_\ep)- \phi(x_\ep,t_\ep^m)}{\ep^2}\ge  \frac{1}2 \langle D^2\phi(x_\ep,t_\ep^m) y_\ep^m, y_\ep^m\rangle + \frac{o(\ep^2)}{\ep^2},
\]
and using  the Taylor expansion of first order of $\phi(x_\ep,\cdot)$ and using \eqref{igxm1}, we get
\begin{equation}\label{dess1}
-\left(\frac{1-c}{c}s_\ep^m-\frac{c+1}{2}\right)\phi_t(x_\ep,t_\ep)\ge  \frac{1}2 \langle D^2\phi(x_\ep,t_\ep^m) y_\ep^m, y_\ep^m\rangle + o(1).
\end{equation}
On the other hand, since $(y_\ep^m,s_\ep^m)\in\B(x_\ep,t_\ep)\subset B(0,1)\times[-\nicefrac{c}2,\nicefrac{c}2]$ for all $\ep>0,$ there exists a subsequence, still denoted by
$\{(y_\ep^m,s_\ep^m)\}_{\ep>0},$ such that
\begin{equation}\label{lim}
(y_\ep^m,s_\ep^m)\to(y_0,s_0)\in B(0,1)\times[-\nicefrac{c}2,\nicefrac{c}2]
\end{equation}
as $\ep\to0^+.$ Moreover $(y_0,s_0)\in\B(x_0,y_0)$ due to {\bf A3}. Thus, taking limit in \eqref{dess1} as $\ep\to0^+,$ we have that
\begin{equation}\label{dess2}
0\le\left(-\frac{1-c}{c}s_0+\frac{c+1}{2}\right)\phi_t(x_0,t_0)-  \frac{1}2 \langle D^2\phi(x_0,t_0) y_0, y_0\rangle.
\end{equation}

In the case that $\nabla \phi (x_0,t_0) = 0,$ we have
\begin{align*}
	0\le& \left(-\frac{1-c}{c}s_0+\frac{c+1}{2}\right)\phi_t(x_0,t_0)-  \frac{1}2 \langle D^2\phi(x_0,t_0) y_0, y_0\rangle\\
	\le&\max_{(y,s)\in\B(x_0,t_0)}\left\{\left(-\frac{1-c}{c}s+\frac{c+1}{2}\right)\phi_t(x_0,t_0)-  \frac{1}2 \langle D^2\phi(x_0,t_0) y, y\rangle  \right\}\\
	=&G^*(D^2\phi(x_0,t_0),\nabla\phi(x_0,t_0),\phi_t(x_0,t_0),x_0,t_0).
\end{align*}
Now, we study the case $\nabla \phi (x_0,t_0) \neq 0.$ We claim that
\[
(y_0,s_0)=(J_{(x,t)}(\nabla\phi(x_0,t_0)),I_{(x,t)}(\nabla\phi(x_0,t_0))).
\]
From this claim and \eqref{dess2}, we have that
\begin{align*}
	0\le& \left(-\frac{1-c}{c}s_0+\frac{c+1}{2}\right)\phi_t(x_0,t_0)-  \frac{1}2 \langle D^2\phi(x_0,t_0) y_0, y_0\rangle\\
	=&\left(-\frac{1-c}{c}s_0+\frac{c+1}{2}\right)\phi_t(x_0,t_0)\\
	&-  \frac{1}2 \langle D^2\phi(x_0,t_0)
	J_{(x,t)}(\nabla\phi(x_0,t_0)), J_{(x,t)}(\nabla\phi(x_0,t_0))\rangle \\
	=&G^*(D^2\phi(x_0,t_0),\nabla\phi(x_0,t_0),\phi_t(x_0,t_0),x_0,t_0).
\end{align*}
Now we prove the claim. First we observe that
\[
\langle \nabla\phi(x_0,t_0), y_0 \rangle \ge \langle \nabla\phi(x_0,t_0), J_{(x_0,t_0)}(\nabla\phi(x_0,t_0))  \rangle
\]
due to $(y_0,t_0)\in\B(x_0,t_0).$

On the other hand, by {\bf A3}, there exists $(y_\ep,s_\ep)\in\B(x_\ep,t_\ep)$ such that
\begin{equation}\label{lim1}
(y_\ep,s_\ep)\to(J_{(x_0,t_0)}(\nabla\phi(x_0,t_0),I_{(x_0,t_0)}(\nabla\phi(x_0,t_0)) \mbox{ as }\ep\to0^+.
\end{equation}
Then,
\[
\phi(z_\ep,r_\ep)\ge\phi(x_\ep^m,t_\ep^m)
\]
where
\[
(z_\ep,r_\ep)=(x_\ep,t_\ep)+ \left(\ep y_\ep,\ep^2\frac{1-c}{c}s_\ep-\ep^2\frac{c+1}2\right).
\]
Using \eqref{lim} and \eqref{lim1} we get
\[
0\le\frac{\phi(z_\ep,r_\ep)-\phi(x_\ep^m,t_\ep^m)}{\ep}\to
\langle\nabla\phi(x_0,t_0),J_{(x_0,t_0)}(\nabla\phi(x_0,t_0))-y_0\rangle
\]
as $\ep\to0^+.$ Thus
\[
\langle \nabla\phi(x_0,t_0), y_0 \rangle = \langle \nabla\phi(x_0,t_0), J_{(x_0,t_0)}(\nabla\phi(x_0,t_0))  \rangle
\]
and by {\bf A4} we have $(y_0,s_0)=(J_{(x_0,t_0)}(\nabla\phi(x_0,t_0)),I_{(x_0,t_0)}(\nabla\phi(x_0,t_0))).$
\end{proof}

\begin{exa}  We now give some examples.
	\begin{enumerate}
			\item If we take the family of sets $\{\B_1(x,t)\}_{(x,t)\in\Omega_T},$ 
			where $\B_1(x,t)$ is defined in Example \ref{ejemplo}, we have
			that
			\[
			G(M,v,s,x,t)= 
			\begin{cases}
				\frac{c+1}{2}s-\frac{\rho^2}{2|v|^2} \langle Mv,v\rangle &\mbox{if } v\neq0,\\[5pt]
				\frac{1-c}{c}\rho|s|  +\frac{c+1}{2} s &\mbox{if } v=0.
			\end{cases}
			\]
		\item Let $\{\B_2(x,t)\}_{(x,t)\in\Omega_T}$ be the family of sets defined in Example \ref{ejemplo}, then
			\[
			G(M,v,s,x,t)= 
			\begin{cases}
				cs-\frac{\rho^2}{2|v|^2} \langle Mv,v\rangle &\mbox{if } v\neq0,\\[5pt]
				\frac{c+1}{2} s &\mbox{if } v=0.
			\end{cases}
			\]
	\end{enumerate}
\end{exa}

\bibliographystyle{amsplain}

\end{document}